\documentclass{aptpub}

\usepackage{bbm}

\newcommand{\mmp}{\mathbb{P}}

\newcommand{\me}{\mathbb{E}}

\newcommand{\mr}{\mathbb{R}}
\newcommand{\mn}{\mathbb{N}}

\newcommand{\lin}{\lim_{n\to\infty}}

\newcommand{\lix}{\underset{x\to\infty}{\lim}}

\newcommand{\iii}{{\rm i}}
\DeclareMathOperator{\1}{\mathbbm{1}}

\authornames{D. Buraczewski, P. Dyszewski, A. Iksanov, A. Marynych} 
\shorttitle{Perpetuities with gamma-like tails} 

\begin{document}

\title{On perpetuities with gamma-like tails} 

\authorone[University of Wroc{\l}aw]{Dariusz Buraczewski} 

\addressone{Mathematical Institute, University of
Wroc{\l}aw, 50-384 Wroc{\l}aw, Poland; e-mail: dbura@math.uni.wroc.pl} 

\authortwo[University of Wroc{\l}aw]{Piotr Dyszewski}

\addresstwo{Mathematical Institute, University of
Wroc{\l}aw, 50-384 Wroc{\l}aw, Poland; e-mail:
pdysz@math.uni.wroc.pl}

\authorthree[Taras Shevchenko National University of Kyiv]{Alexander Iksanov}

\addressthree{Faculty of Computer Science and
Cybernetics, Taras Shevchenko National University of Kyiv, 01601
Kyiv, Ukraine; e-mail: iksan@univ.kiev.ua}

\authorfour[Taras Shevchenko National University of Kyiv]{Alexander Marynych}

\addressfour{Faculty of Computer Science and
Cybernetics, Taras Shevchenko National University of Kyiv, 01601
Kyiv, Ukraine; e-mail: marynych@unicyb.kiev.ua}

\begin{abstract}
An infinite convergent sum of independent and identically
distributed random variables discounted by a multiplicative random
walk is called perpetuity, because of a possible actuarial
application. We give three disjoint groups of sufficient
conditions which ensure that the right tail of a perpetuity
$\mmp\{X>x\}$ is asymptotic to $ax^ce^{-bx}$ as $x\to\infty$ for
some $a,b>0$ and $c\in\mr$. Our results complement those of
Denisov and Zwart [J. Appl. Probab. \textbf{44} (2007),
1031--1046]. As an auxiliary tool we provide criteria for the
finiteness of the one-sided exponential moments of perpetuities.
Several examples are given in which the distributions of
perpetuities are explicitly identified.
\end{abstract}



\section{Introduction} 

Let $(A_n, B_n)_{n\in\mn}$ be a sequence of independent and
identically distributed $\mr^2$-valued random vectors with generic
copy $(A, B)$. Put $\Pi_0:=1$ and $\Pi_n:=A_1\cdot \ldots \cdot
A_n$ for $n\in\mn$. The random discounted sum $$X:=\sum_{k\geq
1}\Pi_{k-1}B_k,$$ provided that $|X|<\infty$ a.s., is called {\it
perpetuity} and is of interest in various fields of applied
probability. The term `perpetuity' stems from the fact that such
random series occur in the realm of insurance and finance as sums
of discounted payment streams. Detailed information about various
aspects of perpetuities, including applications, can be found in
the recent monographs \cite{Buraczewski+Damek+Mikosch:2016,
Iksanov:2017}.

There are a number of papers investigating the asymptotics of
$-\log \mmp\{|X|>x\}$ as $x\to\infty$ in the situations when
$\mmp\{|X|>x\}$ exhibits exponential or superexponential decrease,
see \cite{Alsmeyer+Dyszewski:2016+, Alsmeyer+Iksanov+Roesler:2009,
Goldie+Grubel:1996, Hitczenko:2010, Hitczenko+Wesolowski:2009,
Kolodziejek:2016+}. In the present paper we are interested in
precise (non-logarithmic) asymptotics of $\mmp\{X>x\}$ as
$x\to\infty$. Specifically, our main concern is: which conditions
ensure that $\mmp\{X>x\}\sim a x^c e^{-bx}$ as $x\to\infty$ for
some positive $a$, $b$ and real $c$. Distribution tails which
exhibit such asymptotics may be called {\it gamma-like tails},
hence the title of the paper. To our knowledge, works in this
direction are rare. We are only aware of \cite{Denisov+Zwart:2007,
Konstantinides+Ng+Tang:2010, Maulik+Zwart:2006, Rootzen:1986}. The
first three papers are concerned with exponential tails of
perpetuities which correspond to nonnegative and independent $A$
and $B$. The results obtained in \cite{Rootzen:1986} cover the
situation when $A=\gamma\in (0,1)$ a.s., $B$ is not necessarily
nonnegative and satisfies $\mmp\{B>x\}\sim ax^c e^{-x^p}$ as
$x\to\infty$ for some positive $a$, $p$ and real $c$. Under
additional technical assumptions in the case $p>1$ that paper
points out the asymptotics of $\mmp\{X>x\}$ as $x\to\infty$.

We note that the perpetuities with heavy tails have received much
more attention than those with light tails, \cite{Goldie:1991,
Grey:1994, Grincevicius:1975, Kesten:1973} being classical
articles in the area. A non-exhaustive list of very recent
contributions includes \cite{Damek+Dyszewski:2017,
Damek+Kolodziejek:2017, Dyszewski:2016, Kevei:2016, Kevei:2017}.

\section{Main results}

The following result was given as Proposition 4.1 in
\cite{Denisov+Zwart:2007} under the assumptions that $A$ and $B$
are a.s.\ nonnegative and that $\mmp\{A=1\}=0$ which are partially
dispensed with here. For $s\in\mr$, define $\psi(s):=\me e^{sX}$
and $\varphi(s):=\me e^{sB}$, finite or infinite.
\begin{prop}\label{main}
Let $A$ and $B$ be independent and $r>0$. Suppose that
$\mmp\{A=1\}\in [0,1)$ and that either

\noindent (a) $\mmp\{A\in (0,1]\}=1$ or

\noindent (b) $\mmp\{A\in (-1, 0)\}=1$, or

\noindent (c) $\mmp\{|A|\in (0,1]\}=1$ and $\mmp\{A=-1\}\in
(0,1)$.

\noindent (I) Assume that $\mmp\{B=0\}<1$.

\noindent Let the assumption (a) prevail. If $\mmp\{A=1\}=0$, then
$\me \psi(rA)<\infty$ if, and only if, $\me \varphi(rA)<\infty$.
If $\mmp\{A=1\}\in (0,1)$, then $\me \psi(rA)<\infty$ if, and only
if, $\varphi(r)\mmp\{A=1\}<1$.

\noindent Under the assumption (b) $\me \psi(rA)<\infty$ if, and
only if,
\begin{equation}\label{str}
\me e^{rA_1(B_2+A_2B_3)}<\infty.
\end{equation}

\noindent Under the assumption (c) $\me \psi(rA)<\infty$ if, and
only if,
$$\me e^{-rB}\me e^{rB}\big[\mmp\{A=-1\}\big]^2<\big(1-\me
e^{-rB}\mmp\{A=1\}\big)\big(1-\me e^{rB}\mmp\{A=1\}\big).$$

\noindent (II) Suppose that $\mmp\{B>x\}\sim g(x) e^{-bx}$ as
$x\to\infty$ for some $b>0$ and some function $g$ such that
$g(\log x)$ is slowly varying at $\infty$ and that
$${\lim\sup}_{x\to\infty}({\sup}_{1\leq y\leq
x}\,g(y))/g(x)<\infty.$$ Then
\begin{equation}\label{00}
\mmp\{X>x\}\sim \me \psi(bA) \mmp\{B>x\},\quad x\to\infty
\end{equation}
provided that $\me \psi(bA)<\infty$.
\end{prop}
\begin{rem}
Here is a comment on inequality \eqref{str}. If $B\geq 0$ or
$B\leq 0$ a.s., then \eqref{str} is equivalent to $\me
\varphi(rA_1A_2)<\infty$ and $\me \varphi(rA)<\infty$,
respectively. If $A=-\gamma$ a.s.\ for some $\gamma\in (0,1)$ and
$B$ takes values of both signs with positive probability, then
\eqref{str} is equivalent to $\varphi(-r\gamma)<\infty$ and
$\varphi(r\gamma^2)<\infty$. In the general case, \eqref{str}
which imposes restrictions on both tails of $B$ entails but is not
equivalent to $\me\varphi(rA)<\infty$ and $\me
\varphi(rA_1A_2)<\infty$.
\end{rem}

The argument of \cite{Denisov+Zwart:2007} for part (II) remains
valid in the extended situation treated here. Our contribution
consists in proving part (I), that is, a criterion for $\me
\psi(bA)$ to be finite which is actually a consequence of Theorems
\ref{main_exp}, \ref{main_exp120} and Remark \ref{neg}.

Given next is the more complicated result in which $A$ and $B$ are
allowed to be dependent in a certain way, and the right tail of
possibly two-sided $B$ is gamma-like. Throughout the paper we
shall use the standard notation $x^+:=\max(x,0)$ and
$x^-:=-\min(x,0)$ for $x\in\mathbb{R}$.
\begin{thm}\label{thm:p<1}
Assume that $\mmp\{A\in (0,1]\}=1$;
\begin{equation}\label{cond4}
\mmp\{B>x\}\sim ax^ce^{-bx},\quad x\to\infty
\end{equation}
for some $a,b>0$ and $c<-1$;
\begin{equation}\label{eq:momB1}
\me e^{b B} \1_{\{ A=1\}} < 1;
\end{equation}
\begin{equation}\label{eq:tailB}
\mmp\{Ay+B>x\} \sim f(y) \mmp\{B>x\},\quad x\to\infty
\end{equation}
for each $y\in\mr$ and a nonnegative measurable function $f$;  and
\begin{equation}\label{cond2}
\me \log(1+B^-)<\infty.
\end{equation}
Then $\me f(X)<\infty$ and
\begin{equation}\label{res}
\mmp\{X>x\}~ \sim~ \frac{\me f(X)}{1-\me
e^{bB}\1_{\{A=1\}}}\mmp\{B>x\},\quad x\to\infty.
\end{equation}
\end{thm}

\begin{rem}
Recall that the distribution of a nonnegative random variable $Y$
belongs to the class $\mathcal{S}(\alpha)$ for $\alpha\geq 0$, if
\begin{itemize}
\item[(a)] $\lix \frac{\mmp\{Y>x-y\}}{\mmp\{Y>x\}}=e^{\alpha y}$ for each $y\in\mr$;
\item[(b)] $\lix \frac{\mmp\{Y+Y^\ast>x\}}{\mmp\{Y>x\}}= 2\me e^{\alpha Y}<\infty$, where $Y^\ast$ is an independent copy of $Y$.
\end{itemize}

Condition \eqref{cond4} with $c<-1$ ensures that the distribution
of  $B^+$ belongs to $\mathcal{S}(b)$. While point (a) above is
easily checked, point (b) follows from Lemma 7.1 (iii) in
\cite{Rootzen:1986}. Theorem \ref{thm:p<1} is closely related to
Proposition 4.2 in \cite{Denisov+Zwart:2007} in which a similar
asymptotic result was proved under the assumptions that $A$ and
$B$ are independent, that $\mmp\{A=1\}=0$ and $\mmp\{B\geq 0\}=1$,
and that the distribution of $B$ belongs to the class
$\mathcal{S}(b)$. Theorem 3.2 in
\cite{Konstantinides+Ng+Tang:2010} is another result in this vein.
A perusal of the proof given below reveals that \eqref{res}
remains valid if \eqref{cond4} is replaced by the assumption that
the distribution of $B^+$ belongs to $\mathcal{S}(b)$. However, we
refrain from formulating Theorem \ref{thm:p<1} in this way, for
our focus here is on the gamma-like tails.
\end{rem}

\begin{rem}\label{refe}
Here, we provide more details on functions $f$ arising in
\eqref{eq:tailB} assuming that the assumptions of Theorem
\ref{thm:p<1} are in force. It is clear that $f(y)=\me e^{byA}$,
$y\in\mr$ whenever $A$ and $B$ are independent. The last equality
is not necessarily true when $A$ and $B$ are dependent. For
instance, if $A=\zeta_1\1_{\{B>q\}}+\zeta_2\1_{\{B\leq q\}}$ for
some $\zeta_1,\zeta_2\in (0,1)$, $\zeta_1\neq \zeta_2$ and some
$q>0$, then $f(y)=e^{by\zeta_1}\neq \me e^{byA}$, $y\in\mr$.

We note that a condition of form \eqref{eq:tailB} appears in
Theorem 3 of \cite{Palmowski+Zwart:2007} in the setting quite
different from ours. The cited result gives sufficient conditions
under which the right tail of $\sup_{k\geq 1}\Pi_{k-1}B_k$ is
heavy. One of the referees has kindly informed us that our method
of proof of Theorem \ref{thm:p<1} is rather similar to that of
Theorem 3 in \cite{Palmowski+Zwart:2007}. More details on this
point will be given at the end of Section \ref{12}.
\end{rem}

Proposition \ref{main} and Theorem \ref{thm:p<1} cover the
situation where a gamma-like tail of $X$ is inherited from a
gamma-like tail of $B$, the influence of the distribution of $A$
being small, for it is only seen in the multiplicative constant.
Example \ref{basic} given below reveals that the distributions of
both $A$ and $B$ may give principal contributions to a gamma-like
tail of $X$.

To proceed we need more notation. Denote by $\gamma(a,b)$ and
$\beta(c,d)$ a gamma distribution with parameters $a,b>0$ and a
beta distribution with parameters $c,d>0$, respectively. Recall
that $$\gamma(a,b)({\rm
d}x)=\frac{b^ax^{a-1}e^{-bx}}{\Gamma(a)}\1_{(0,\infty)}(x){\rm
d}x,$$ where $\Gamma(\cdot)$ is the Euler gamma function, and
$$\beta(c,d)({\rm d}x)=\frac{1}{{\rm B}(c,d)}x^{c-1}(1-x)^{d-1}\1_{(0,1)}(x){\rm d}x,$$ where ${\rm B}(\cdot,\cdot)$ is the Euler beta function. The
following example is well-known, see, for instance, Example 3.8.2
in \cite{Vervaat:1979}.
\begin{ex}\label{basic} \rm
Assume that $A$ and $B$ are independent, $A$ has a $\beta(c,1)$
distribution and $B$ has a $\gamma(1,b)$ (exponential)
distribution. Then $X$ has a $\gamma(c+1,b)$
distribution\footnote{This can be checked in several ways, for
instance, via the argument given in Example \ref{2}}. In
particular,
\begin{equation}\label{asy}
\mmp\{X>x\}\sim \frac{(bx)^c} {\Gamma(c+1)}e^{-bx},\quad
x\to\infty.
\end{equation}
\end{ex}

Our next result, Theorem \ref{main2}, provides an extension of
Example \ref{basic} in that $B$ is allowed to take values of both
signs with positive probability and that the right tail of $B$ is
approximately, rather than precisely, exponential. Our Theorem
\ref{main2} is close in spirit to Theorem 6.1 in
\cite{Maulik+Zwart:2006} because in both results it is assumed
that while one of the independent input random variables $A$ and
$B$ obeys a particular distribution ($A$ has a $\beta(1,\lambda)$
distribution in our Theorem \ref{main2}; $B$ has a
$\gamma(1,\lambda)$ distribution in Theorem 6.1
\cite{Maulik+Zwart:2006}), the distribution of the other random
variable follows a prescribed tail behavior.
\begin{thm}\label{main2}
Assume that $A$ and $B$ are independent; $A$ has a
$\beta(\lambda,1)$ distribution for some $\lambda>0$; condition
\eqref{cond2} holds and
\begin{equation}\label{dec1}
\mmp\{B>x\}=Ce^{-bx}+r(x)
\end{equation}
for some $C,b>0$, all $x\geq 0$, and a function $r$ such that
\begin{equation}\label{cond1}
\lix e^{bx}r(x)=0,
\end{equation}
\begin{equation}\label{r_1_integreable}
\int_1^\infty\frac{e^{by}}{y}r^{+}(y){\rm d
y}<\infty\quad\text{and}\quad\int_1^\infty\frac{e^{(b+\varepsilon)y}}{y}r^{-}(y){\rm
d y}<\infty
\end{equation}
for some $\varepsilon>0$. Then
\begin{equation}\label{main_claim}
\mmp\{X>x\}\sim Kx^{\lambda C}e^{-bx},\quad x\to\infty,
\end{equation}
where
$$K:=\frac{C b^{C\lambda}}{\Gamma(C\lambda+1)}\exp\left(\lambda\left[\int_0^\infty \frac{e^{by}-1}{y}r(y){\rm
d}y-\int_{-\infty}^0\frac{e^{by}-1}{y}\mmp\{B\leq y\}{\rm
d}y\right]\right)<\infty.$$
\end{thm}

The remainder of the paper is organized as follows. In Section
\ref{ex} we give several examples intended to illustrate
Proposition \ref{main} and Theorem \ref{main2}. Also in this
section is a discussion of an interesting connection between
perpetuities arising in Theorem \ref{main2} and certain
selfdecomposable distributions. It is exactly this link which
makes the proof of Theorem \ref{main2} relatively simple. In
Section \ref{crit} we provide criteria for the existence of the
one-sided exponential moments of perpetuities, the results which
are needed for the proof of Proposition \ref{main}. The picture is
incomplete yet, for a criterion remains a challenge in the case
where both $A$ and $B$ take values of both signs with positive
probability. All the proofs are given in Sections \ref{11},
\ref{12} and \ref{13}.

\section{Illustrating examples}\label{ex}

Here is an example illustrating Proposition \ref{main}.
\begin{ex} \rm
Denote by $\theta_a$ and $\theta_b$ independent random variables
with a $\gamma(1,a)$ and $\gamma(1,b)$ distribution, respectively.
Let $A=\gamma\in (0,1)$ a.s.\ and $\me e^{sB}=\frac{a-\gamma
s}{a-s}\frac{b+\gamma s}{b+s}$ for $-b<s<a$. Then
$X=\theta_a-\theta_b$ or equivalently $\me
e^{sX}=\frac{a}{a-s}\frac{b}{b+s}$. A standard calculation shows
that $\mmp\{X>x\}=\frac{b}{a+b}e^{-ax}$ for $x>0$. The
distribution of $B$ is a mixture of the atom at zero with weight
$\gamma^2$, the distribution of $\theta_a$ with weight
$\gamma(1-\gamma)$, the distribution of $-\theta_b$ with weight
$\gamma(1-\gamma)$ and the distribution of $\theta_a-\theta_b$
with weight $(1-\gamma)^2$. Hence,
$\mmp\{B>x\}=(1-\gamma)\frac{b+a\gamma}{a+b}e^{-ax}$ for $x>0$ in
full agreement with Proposition \ref{main}.
\end{ex}

It is well known that the explicit distributions of the
perpetuities are rarely available. Below we give several examples
of distributions of $B$ satisfying the assumptions of Theorem
\ref{main2} for which distributions of the corresponding
perpetuities $X$ can be identified. Among others, this allows us
to check validity of formula \eqref{main_claim}. We start with a
trivial observation that the distributions of $A$ and $B$ as given
in Example \ref{basic} satisfy the assumptions of Theorem
\ref{main2} with $\lambda=c$, $C=1$ and $r(x)\equiv 0$ in which
case \eqref{main_claim} amounts to \eqref{asy} as it must be.

Throughout the rest of the section we assume, without further
notice, that $B$ is independent of $A$ and that $A$ has a
$\beta(1,\lambda)$ distribution. We first point out an interesting
connection with special selfdecomposable distributions which
enables us to obtain a useful representation
\begin{equation}\label{chf}
\Psi(t):=\me e^{{\rm i}tX}=\Phi(t)\exp\left(\lambda\int_0^t
\frac{\Phi(u)-1}{u}{\rm d}u\right),\quad t\in\mr,
\end{equation}
where $\Phi(t):=\me e^{{\rm i}tB}$, $t\in\mr$. The connection is
implicit in \cite{Takacz:1955, Vervaat:1979} and perhaps some
other works.

The class $L$ of selfdecomposable distributions is comprised of
all possible limit distributions for the sums, properly normalized
and centered, of independent (not necessarily identically
distributed) random variables satisfying an infinitesimality
condition. It was proved in \cite{Jurek+Vervaat:1983} that the
class $L$ coincides with the class of distributions of the random
variables $J:=\int_{(0,\infty)}e^{-s}{\rm d}Y(s)$, where
$(Y(t))_{t\geq 0}$ is a L\'{e}vy process with $\me \log
(1+|Y(1)|)<\infty$. It is known (see, for instance, formula (4.4)
in \cite{Jurek+Vervaat:1983}) that $$\log\me e^{{\rm
i}tJ}=\int_0^t\frac{\log\me e^{{\rm i}sY(1)}}{s}{\rm d}s,\quad
t\in\mr.$$ If $(Y(t))_{t\geq 0}$ is a compound Poisson process of
intensity $\lambda$ with jumps $B_k$ satisfying the assumptions of
Theorem \ref{main2}, then
\begin{equation}\label{chf2}
\log\me e^{{\rm i}tJ}=\lambda \int_0^t\frac{\Phi(s)-1}{s}{\rm
d}s,\quad t\in\mr
\end{equation}
as a consequence of $\log\me e^{{\rm i}tY(1)}=\lambda (\Phi(t)-1)$
for $t\in\mr$. Recalling that the function $x\mapsto \log(1+x)$ is
subadditive on $[0,\infty)$ we conclude that conditions
\eqref{cond2} and \eqref{dec1} ensure that $\me \log(1+|B|)\leq
\me \log(1+B^+)+\me \log (1+B^-)<\infty$, whence $\me \log
(1+|Y(1)|)\leq \me N\me \log (1+|B|)<\infty$, where $N$ is a
Poisson distributed random variable with parameter $\lambda$. The
latter inequality secures the convergence of the integral in
\eqref{chf2}. The selfdecomposable distributions with the
characteristic functions of form \eqref{chf2} were investigated in
\cite{Iksanov:2002, Iksanov+Jurek:2003}. Formula \eqref{chf} is a
consequence of \eqref{chf2} and a representation
$X=B_1+A_1(B_2+A_2B_3+\ldots)$ a.s.\ and the fact that
$A_1(B_2+A_2B_3+\ldots)$ is independent of $B_1$ and has the same
distribution as $J$ in \eqref{chf2}.

\begin{ex}\label{2} \rm
Let $B=\xi/b-\eta/a$ for $a,b>0$ and independent random variables
$\xi$ and $\eta$ with a $\gamma(1,1)$ distribution (exponential
distribution of unit mean). Then
$\mmp\{B>x\}=\frac{a}{a+b}e^{-bx}$ for $x>0$ and
\begin{equation}\label{cond3}
\mmp\{B\leq x\}=\frac{b}{a+b}e^{ax}~\text{for}~ x<0,
\end{equation}
so that the assumptions of Theorem \ref{main2} are satisfied with
$C=a/(a+b)$ and $r(x)\equiv 0$. Since $$\Phi(t)=\me e^{{\rm
i}tB}=\frac{b}{b-{\rm i}t}\frac{a}{a+{\rm i}t},\quad t\in\mr,$$ we
infer with the help of \eqref{chf}
$$\me e^{{\rm i}tX}=\bigg(\frac{b}{b-{\rm i}t}\bigg)^{\frac{a\lambda}{a+b}+1}\bigg(\frac{a}{a+{\rm i}t}\bigg)^{\frac{b\lambda}{a+b}+1},\quad t\in\mr.$$ Thus, $X$ has the same distribution as $Y-Z$, where $Y$ and $Z$
are independent random variables with $\gamma(a\lambda/(a+b)+1,b)$
and $\gamma(b\lambda/(a+b)+1,a)$ distributions, respectively.
Noting that the function $x\mapsto\mmp\{e^Y>x\}$ is regularly
varying at $\infty$ of index $-b$ and applying Breiman's lemma
(Proposition 3 in \cite{Breiman:1965} and Corollary 3.6 (iii) in
\cite{Cline+Samorodnitsky:1994}) we conclude that
$$\mmp\{X>x\}=\mmp\{e^Ye^{-Z}>e^x\}\sim \me
e^{-bZ}\mmp\{Y>x\}=\bigg(\frac{a}{a+b}\bigg)^{\frac{b\lambda}{a+b}+1}\mmp\{Y>x\},\quad
x\to\infty.$$ In view of \eqref{asy} this entails
\begin{equation}\label{3}
\mmp\{X>x\}\sim
\bigg(\frac{a}{a+b}\bigg)^{\frac{b\lambda}{a+b}+1}\frac{(bx)^{\frac{a\lambda}{a+b}}}{\Gamma(a\lambda/(a+b)+1)}e^{-bx},\quad
x\to\infty.
\end{equation}
To check that formula \eqref{main_claim} gives the same answer we
have to calculate $K$ appearing in that formula. Using
\eqref{cond3} we obtain
\begin{eqnarray}
&&\exp\bigg(-\lambda\int_{-\infty}^0\frac{e^{by}-1}{y}\mmp\{B\leq
y\}{\rm
d}y\bigg)\notag\\&=&\exp\bigg(-\frac{b\lambda}{a+b}\int_0^\infty\frac{e^{-ay}-e^{-(a+b)y}}{y}{\rm
d}y\bigg)\notag\\&=&\exp\bigg(-\frac{b\lambda}{a+b}\log\bigg(\frac{a+b}{a}\bigg)\bigg)=\bigg(\frac{a}{a+b}\bigg)^{\frac{b\lambda}{a+b}}\label{fr}
\end{eqnarray}
having observed that the last integral is a Frullani integral.
Thus,
$$K=\frac{\frac{a}{a+b}b^{\frac{a\lambda}{a+b}}}{\Gamma(a\lambda/(a+b)+1)}\bigg(\frac{a}{a+b}\bigg)^{\frac{b\lambda}{a+b}}
=\bigg(\frac{a}{a+b}\bigg)^{\frac{b\lambda}{a+b}+1}\frac{b^{\frac{a\lambda}{a+b}}}{\Gamma(a\lambda/(a+b)+1)}$$
which is in line with \eqref{3}.
\end{ex}

\begin{ex}\rm
Put $B:=\xi-\eta$ for independent positive random variables $\xi$
and $\eta$. Assume that
\begin{equation}\label{inter}
\mmp\{\xi>x\}=C_1e^{-bx}+r_1(x),\quad x\geq 0
\end{equation}
and that $r_1$ satisfies \eqref{cond1} and
\eqref{r_1_integreable}. Then
\begin{align*}
\mmp\{B > x\}&=\int_0^{\infty}\mmp\{\xi>x+y\}\mmp\{\eta\in{\rm d}y\}\\
&=C_1(\me e^{-b\eta})e^{-bx}+\me r_1(x+\eta)=:C e^{-bx}+r(x).
\end{align*}
By the Lebesgue dominated convergence theorem $\lix e^{bx}r(x)=0$.
Furthermore, by Fubini's theorem and the fact that $y\mapsto
y^{-1}e^{by}$ is nondecreasing on $[1/b,\infty)$ we obtain
\begin{multline*}
\int_{1/b}^\infty\frac{e^{by}}{y}r^{+}(y){\rm d}y\leq \me
\int_{1/b}^\infty \frac{e^{by}}{y}r_1^{+}(y+\eta){\rm d}y \leq\\
\me \int_{1/b}^\infty
\frac{e^{b(y+\eta)}}{y+\eta}r_1^{+}(y+\eta){\rm d}y =\me
\int_{1/b+\eta}^\infty \frac{e^{by}}{y}r_1^{+}(y){\rm d}y \leq
\int_{1/b}^\infty \frac{e^{by}}{y}r_1^{+}(y){\rm d}y<\infty.
\end{multline*}
Analogously,
$$
\int_1^\infty \frac{e^{(b+\varepsilon)y}}{y}r^{-}(y){\rm
d}y<\infty.
$$
Hence, under \eqref{inter} the right tail of the distribution of
$B$ satisfies the assumptions of Theorem \ref{main2} with
$C:=C_1\me e^{-b\eta}$ and $r(x):=\me r_1(x+\eta)$ whatever the
distribution of $\eta$.

To give a concrete example let $\xi$ and $\eta$ be independent
with $\mmp\{\xi>x\}=\mmp\{\eta>x\}=pe^{-bx}+(1-p)e^{-cx}$ for
$x\geq 0$, $c>b>0$ and $p\in (0,1)$. Condition \eqref{inter} holds
with $C_1=p$ and $r_1(x)=(1-p)e^{-cx}$ which trivially satisfies
\eqref{cond1} and \eqref{r_1_integreable}. Further,
\begin{equation}\label{aux}
\mmp\{B>x\}=\mmp\{B\leq -x\}=\frac{c_1e^{-bx}+c_2e^{-cx}}{2},\quad
x\geq 0,
\end{equation}
where
$$c_1:=p^2+\frac{2p(1-p)c}{b+c},~c_2:=(1-p)^2+\frac{2p(1-p)b}{b+c},$$
which immediately implies that condition \eqref{cond2} holds and
that $B=\xi-\eta$ has the characteristic function
\begin{equation*}
\Phi(t)=\me e^{\iii t
B}=c_1\frac{b^2}{b^2+t^2}+c_2\frac{c^2}{c^2+t^2},\quad t\in\mr.
\end{equation*}
Observing that $$\exp\bigg(\alpha\int_0^\infty (e^{{\rm
i}ut}-1)\frac{e^{-\beta u}}{u}{\rm
d}u\bigg)=\bigg(\frac{\beta}{\beta-{\rm i}t}\bigg)^\alpha,\quad
t\in\mr$$ for $\alpha,\beta>0$ we obtain, with the help of
\eqref{aux},
\begin{eqnarray*}
&&\exp\bigg(\lambda \int_0^t\frac{\Phi(u)-1}{u}{\rm
d}u\bigg)\\&=&\exp\bigg(\lambda\int_0^\infty (e^{{\rm
i}ut}-1)\frac{\mmp\{B>u\}}{u}{\rm d}u\bigg)\exp\bigg(\lambda
\int_0^\infty (e^{-{\rm i}ut}-1)\frac{\mmp\{-B>u\}}{u}{\rm
d}u\bigg)\\&=&\left(\frac{b^2}{b^2+t^2}\right)^{c_1\lambda/2}\left(\frac{c^2}{c^2+t^2}\right)^{c_2\lambda/2}.
\end{eqnarray*}
This entails $$ \me e^{\iii t
X}=\Phi(t)\left(\frac{b^2}{b^2+t^2}\right)^{c_1\lambda/2}\left(\frac{c^2}{c^2+t^2}\right)^{c_2\lambda/2}$$
from which we conclude that $X$ has the same distribution as
$\xi-\eta+Y_1-Y_2+Z_1-Z_2$, where the latter random variables are
independent, $Y_1$ and $Y_2$ have a $\gamma(c_1\lambda/2,b)$
distribution, and $Z_1$ and $Z_2$ have a $\gamma(c_2\lambda/2,c)$
distribution. Note that $$\me
e^{-b\eta}=\frac{p}{2}+\frac{(1-p)c}{b+c}=\frac{c_1}{2p},~ \me
e^{-bY_2}=\bigg(\frac{1}{2}\bigg)^{c_1\lambda/2},~ \me
e^{b(Z_1-Z_2)}=\bigg(\frac{c^2}{c^2-b^2}\bigg)^{c_2\lambda/2}$$
and that the exponential moments of order $b+\varepsilon$ for
$\varepsilon\in (0,c-b)$ are finite. Invoking Breiman's lemma
yields
\begin{eqnarray*}
\mmp\{X>x\}&\sim& \me
e^{b(-\eta-Y_2+Z_1-Z_2)}\mmp\{\xi+Y_1>x\}\\&=&\frac{c_1}{2p}\left(\frac{1}{2}\right)^{c_1\lambda/2}\left(\frac{c^2}{c^2-b^2}\right)^{c_2\lambda/2}\mmp\{\xi+Y_1>x\}
\end{eqnarray*}
 as $x\to\infty$. In view of the equality
$\gamma(c_1\lambda/2,b)\ast\gamma(1,b)=\gamma(c_1\lambda/2+1,b)$
and the asymptotic relation
$$\gamma(c_1\lambda/2,b)\ast\gamma(1,c)((x,\infty))=o(\gamma(c_1\lambda/2+1,b)((x,\infty))),\quad x\to\infty$$
we have
$$\mmp\{\xi+Y_1>x\} \sim p\,\gamma(c_1(\lambda/2)+1,b)((x,\infty)),\quad x\to\infty.$$
Combining pieces together and applying formula \eqref{asy} we
obtain
\begin{equation}
\mmp\{X>x\}\sim\frac{c_1}{2}\left(\frac{1}{2}\right)^{c_1\lambda/2}\left(\frac{c^2}{c^2-b^2}\right)^{c_2\lambda/2}\frac{b^{\lambda
c_1/2}}{\Gamma(c_1(\lambda /2)+1)}x^{c_1\lambda/2}e^{-bx},\quad
x\to\infty.\label{asy2}
\end{equation}

Let us show that asymptotics \eqref{asy2} follows from Theorem
\ref{main2} with $C=c_1/2$ and $r(x)=(c_2/2)e^{-cx}$. To this end,
we only have to calculate $K$ appearing in \eqref{main_claim}.
Using a formula for Frullani's integrals (see \eqref{fr}) we
obtain
\begin{eqnarray*}
K&=&\frac{c_1}{2}\frac{b^{c_1\lambda/2}}{\Gamma(c_1(\lambda/2)+1)}\exp\bigg[\lambda\bigg(\int_0^\infty\frac{e^{by}-1}{y}\frac{c_1}{2}e^{-by}{\rm
d}y\\&-&\int_0^\infty
\frac{1-e^{-by}}{y}\bigg(\frac{c_1}{2}e^{-by}+\frac{c_2}{2}e^{-cy}\bigg){\rm
d}y\bigg)\bigg]\\&=&
\frac{c_1}{2}\left(\frac{1}{2}\right)^{c_1\lambda/2}\left(\frac{c^2}{c^2-b^2}\right)^{c_2\lambda/2}\frac{b^{\lambda
c_1/2}}{\Gamma(c_1(\lambda/2)+1)}
\end{eqnarray*}
which is in agreement with \eqref{asy2}.
\end{ex}

\begin{ex}\rm
Let $B$ be a positive random variable with the distribution tail
$$ \mmp\{B>x\}=\frac{1}{\lambda}\frac{e^{-bx}(1-e^{-\lambda
x})}{1-e^{-x}},\quad x>0,
$$
where $b,\lambda>0$ and $2b+\lambda>1$. The last assumption
warrants that the right-hand side is a decreasing function.
Writing $$
\mmp\{B>x\}=\frac{1}{\lambda}e^{-bx}+\frac{1}{\lambda}\frac{e^{-bx}(e^{-x}-e^{-\lambda
x})}{1-e^{-x}}=:Ce^{-bx}+r(x)$$ we conclude that if $\lambda>1$,
then $r^+(x)=r(x)\to (\lambda-1)/\lambda$ as $x\to 0+$ and
$r^+(x)=O(e^{-(b+1)x})$ as $x\to\infty$, whereas if $\lambda\in
(0,1)$, then $r^-(x)=-r(x)\to (1-\lambda)/\lambda$ as $x\to 0+$
and $r^-(x)=O(e^{-(b+\lambda)x})$ as $x\to\infty$. Thus, in both
cases conditions \eqref{cond1} and \eqref{r_1_integreable} are
satisfied.

Let $Y$ be a random variable which is independent of $B$ and has a
$\beta(b,\lambda)$ distribution. It can be checked that
$$
\me e^{-{\rm i}tY}= \frac{\Gamma(b-\iii
t)\Gamma(b+\lambda)}{\Gamma(b)\Gamma(b+\lambda-\iii t)},\quad
t\in\mr.$$ On the other hand, formula 3.413(1) in
\cite{Gradshteyn+Ryzhik:2000} yields
\begin{multline*}
\exp\bigg(\lambda\int_0^t\frac{\Phi(u)-1}{u}{\rm
d}u\bigg)=\exp\bigg(\lambda\int_0^{\infty}\frac{e^{\iii u
t}-1}{u}\mmp\{B>u\}{\rm d}u\bigg)\\
=\exp\bigg(\int_0^\infty \frac{e^{\iii u
t}-1}{u}\frac{e^{-bu}(1-e^{-\lambda u})}{1-e^{-u}}{\rm
d}u\bigg)=\frac{\Gamma(b-\iii
t)\Gamma(b+\lambda)}{\Gamma(b)\Gamma(b+\lambda-\iii t)},
\end{multline*}
whence
$$
\exp\bigg(\lambda\int_0^t\frac{\Phi(u)-1}{u}{\rm d}u\bigg)=\me
e^{-{\rm i}tY}.$$ This representation can be read off from Example
9.2.3 in \cite{Bondesson:1992}, but both the setting and the proof
given in \cite{Bondesson:1992} are slightly different from ours.
Using \eqref{chf} we conclude that $X$ has the same distribution
as $-\log Y+B$. This representation enables us to find the
asymptotics
\begin{eqnarray*}
\mmp\{X>x\}&=&\mmp\{-\log Y>x\}+\mmp\{-\log Y+B>x, -\log Y\leq
x\}\\&=&o(xe^{-bx})+\frac{1}{\lambda {\rm B}(b,\lambda)}\int_0^x
e^{-b(x-y)}e^{-by}(1-e^{-y})^{\lambda-1}{\rm d}y\sim
\frac{1}{\lambda {\rm B}(b,\lambda)}xe^{-bx}
\end{eqnarray*}
as $x\to\infty$. An application of Theorem \ref{main2} in
combination with already used formula 3.413(1) in
\cite{Gradshteyn+Ryzhik:2000} gives the same asymptotics. We omit
details.
\end{ex}

\section{Criteria for the finiteness of the one-sided exponential
moments}\label{crit}

Throughout the rest of the paper we shall often assume that the
following nondegeneracy conditions hold:
\begin{equation}\label{degen1}
\mmp\{A=0\}=0~\text{and}~\mmp\{B=0\}<1
\end{equation}
and
\begin{equation}\label{degen2}
\mmp\{B+Ac=c\}<1~\text{for all}~c\in\mr.
\end{equation}
Also, we shall make a repeated use of the following well known
decomposition
\begin{eqnarray}\label{decomp}
X&=&B_1+A_1B_2+\ldots+A_1\cdot\ldots\cdot
A_{\tau-1}B_\tau+A_1\cdot\ldots\cdot
A_\tau(B_{\tau+1}+A_{\tau+1}B_{\tau+2}+\ldots)\notag\\&=:&X_\tau+\Pi_\tau
X^{(\tau)},
\end{eqnarray}
where $\tau\geq 1$ is either deterministic or a stopping time
w.r.t.\ the filtration generated by $(A_k, B_k)_{k\in\mn}$.
Observe that $X^{(\tau)}=B_{\tau+1}+A_{\tau+1}B_{\tau+2}+\ldots$
has the same distribution as $X$ and is independent of $(\Pi_\tau,
X_\tau)$. This particularly shows that $X$ is a perpetuity
generated by $(\Pi_\tau, X_\tau)$.

Some of our subsequent arguments will rely upon Proposition
\ref{AIR} given below which is a criterion for the finiteness of
$\me e^{r|X|}$. Parts (a) and (b) of Proposition \ref{AIR} are
Theorems 1.6 and 1.7 in \cite{Alsmeyer+Iksanov+Roesler:2009},
respectively.
\begin{prop}\label{AIR}
(a) Suppose \eqref{degen1}, \eqref{degen2} and $\mmp\{|A|=1\}=0$,
and let $r>0$. Then $\me e^{r|X|}<\infty$ if, and only if,
$$\mmp\{|A|<1\}=1~\text{and}~\me e^{r|B|}<\infty.$$

\noindent (b) Suppose \eqref{degen1}, \eqref{degen2} and
$\mmp\{|A|=1\}\in (0,1)$, and let $r>0$. Then $\me
e^{r|X|}<\infty$ if, and only if, $$\mmp\{|A|\leq 1\}=1,\quad \me
e^{r|B|}<\infty$$ and $$\me e^{-rB}\1_{\{A=-1\}}\me
e^{rB}\1_{\{A=-1\}}<(1-\me e^{-rB}\1_{\{A=1\}})(1-\me
e^{rB}\1_{\{A=1\}}).$$
\end{prop}

Next, we provide necessary and sufficient conditions for the
finiteness of the one-sided moments $\me e^{rX}$ which is a
somewhat more delicate problem. First, we state a criterion for
positive $A$.

\begin{thm}\label{main_exp}
Suppose \eqref{degen1}, \eqref{degen2}, $\mmp\{A>0\}=1$,
$|X|<\infty$ a.s., and let $r>0$. The conditions
\begin{equation}\label{boundA}
\mmp\{A\leq 1\}=1,
\end{equation}
\begin{equation}\label{expB}
\me e^{rB}<\infty\quad \text{and}\quad \me e^{rB}\1_{\{A=1\}}<1
\end{equation}
are sufficient for
\begin{equation}\label{expx}
\me e^{rX}<\infty
\end{equation}
to hold.

Conversely, if the support of the distribution of $X$ is unbounded
from the right, then \eqref{expx} entails \eqref{boundA} and
\eqref{expB}, whereas if the support of the distribution of $X$ is
bounded from the right, then $\me e^{sB}<\infty$ for all $s>0$.
\end{thm}
\begin{rem}
As far as condition \eqref{boundA} is concerned, the assumption
about unboundedness of the support of the distribution of $X$ is
indispensable. For a trivial counterexample, just take a.s.\
nonpositive $B$, so that $X\in [-\infty, 0]$ a.s. Then $\me
e^{rX}<\infty$ for each $r>0$, irrespective of whether
$\mmp\{A>1\}$ is positive or equals zero. More interestingly, the
support of the distribution of $X$ can be bounded from the right
even if $\mmp\{B>0, A\neq 1\}>0$ and $\mmp\{A>1\}>0$. Indeed,
assume that the last two inequalities hold true, that
$\mmp\{A>0\}=1$ and that $$\Pi_{\tau}m+X_\tau=A_1\cdot\ldots\cdot
A_\tau m+ B_1+\ldots+B_\tau \le m\quad \text{a.s.}$$ for some real
$m$, where $\tau:=\inf\{k\in\mn: \Pi_k\neq 1\}$ (here, we have
used decomposition \eqref{decomp} with the particular $\tau$).
Then $X\le m$ a.s. (see Lemma 2.5.7 and Figure 2.4(c) in
\cite{Buraczewski+Damek+Mikosch:2016}) whence $\me e^{rX}<\infty$
for each $r>0$ yet $\mmp\{A>1\}>0$.
\end{rem}
\begin{rem}\label{impo}
A perusal of the proof of Theorem \ref{main_exp} reveals that $\me
e^{rX}<\infty$ in combination with $\mmp\{A\in (0,1]\}=1$ entails
$\me e^{rB}\1_{\{A=1\}}<1$, irrespective of whether the support of
the distribution of $X$ is bounded or not.
\end{rem}

\begin{rem}\label{neg}
Passing to the case where $A$ is negative with positive
probability we first single out a simpler situation in which
$\mmp\{A=-1\}>0$. Then $\me e^{rX}<\infty$ if, and only if, $\me
e^{r|X|}<\infty$. Assume that $\psi(r)=\me e^{rX}<\infty$.
Decomposition \eqref{decomp} with $\tau=1$ is equivalent to
\begin{equation}\label{equ}
\psi(r)=\me e^{rB}\psi(rA).
\end{equation}
Now we use \eqref{equ} to obtain
$$\psi(r)=\me e^{rB}\psi(rA)\geq \me e^{rB}\1_{\{A=-1\}}\psi(-r)$$
which shows that $\psi(-r)<\infty$ whence $\me e^{r|X|}\leq
\psi(r)+\psi(-r)<\infty$. This proves the $\Rightarrow$
implication, the implication $\Leftarrow$ being trivial. Thus,
whenever $\mmp\{A=-1\}>0$ a criterion for the finiteness of $\me
e^{rX}$ coincides with that for the finiteness $\me e^{r|X|}$. The
latter is given in Proposition \ref{AIR}.
\end{rem}

When $A$ takes values of both signs with positive probability and
$\mmp\{A=-1\}=0$ we can only prove a criterion under the
additional assumption that $B$ is a.s.\ nonnegative.
\begin{thm}\label{main_exp120}
Suppose \eqref{degen1}, \eqref{degen2}, $\mmp\{A=-1\}=0$,
$|X|<\infty$ a.s., and let $r>0$.

\noindent Assume that $\mmp\{A<0\}\mmp\{A>0\}>0$ and $\mmp\{B\geq
0\}=1$. Then \eqref{expx} holds if, and only if,
\begin{equation}\label{boundAA}
\mmp\{|A|\leq 1\}=1
\end{equation}
and condition \eqref{expB} holds.

\noindent Assume that $\mmp\{A<0\}=1$. Then \eqref{expx} holds if,
and only if, condition \eqref{boundAA} holds and
\begin{equation}\label{bound_compl}
\me e^{r(B_1+A_1B_2)}<\infty.
\end{equation}
\end{thm}

\section{Proofs of Theorems \ref{main_exp} and \ref{main_exp120}, and Proposition
\ref{main}}\label{11}

\begin{proof}[Proof of Theorem \ref{main_exp}]

\noindent {\sc Proof of \eqref{boundA}, \eqref{expB} $\Rightarrow$
\eqref{expx}}. Assume first that $A\in (0,1)$ a.s., i.e.,
$\mmp\{A=1\}=0$, so that we have to show that $\me e^{rB}<\infty$
entails $\me e^{rX}<\infty$ or, equivalently, that
\begin{equation}\label{impl1}
\me e^{rB^+}<\infty~\Rightarrow~ \me e^{rX^+}<\infty.
\end{equation}

Since the function $x\mapsto x^+$ is subadditive on $\mr$ and
satisfies $(\alpha x)^+=\alpha x^+$ for $\alpha>0$ and $x\in\mr$
we infer
$$\exp\,[rX^+]=\exp\bigg[r\bigg(\sum_{k\geq 1}\Pi_{k-1}B_k\bigg)^+\bigg]\leq \exp\bigg[r\sum_{k\geq 1}\Pi_{k-1}B_k^+\bigg] =:\exp[rX^\ast].$$
The random variable $X^\ast\geq 0$ is a perpetuity generated by
$(A, B^+)$. Hence, by Proposition \ref{AIR} $\me e^{rB^+}<\infty$
entails $\me e^{rX^\ast}<\infty$ and thereupon \eqref{impl1}.

Assuming that $A\in (0,1]$ a.s.\ and that $\mmp\{A=1\}\in (0,1)$
we must check that $\me e^{rB}<\infty$ together with $\me
e^{rB}\1_{\{A=1\}}<1$ guarantees $\me e^{rX}<\infty$. Put
$\widehat{T}_0:=0$, $\widehat{T}_k:=\inf\{n>\widehat{T}_{k-1}:
A_n<1\}$ for $k\in\mn$ and then
$$\widehat{A}_k:=A_{\widehat{T}_{k-1}+1}\cdot\ldots\cdot
A_{\widehat{T}_k},\quad
\widehat{B}_k=B_{\widehat{T}_{k-1}+1}+A_{\widehat{T}_{k-1}+1}B_{\widehat{T}_{k-1}+2}+\ldots+A_{\widehat{T}_{k-1}+1}\cdot\ldots\cdot
A_{\widehat{T}_k-1}B_{\widehat{T}_k} $$ for $k\in\mn$. The vectors
$(\widehat{A}_1,\widehat{B}_1)$, $(\widehat{A}_2,
\widehat{B}_2),\ldots$ are independent and identically distributed
and $X=\widehat{B}_1+\sum_{n\geq 1}\widehat{A}_1\cdot\ldots\cdot
\widehat{A}_{n-1}\widehat{B}_n$. Since
\begin{eqnarray*}
\me e^{r\widehat{B}_1}&=&\sum_{n\geq 1}\me
e^{r(B_1+A_1B_2+\ldots+A_1\cdot\ldots\cdot
A_{n-1}B_n)}\1_{\{A_1=\ldots=A_{n-1}=1, A_n<1\}}\\&=& \me
e^{rB}\1_{\{A<1\}}\sum_{n\geq 1}\big(\me
e^{rB}\1_{\{A=1\}}\big)^{n-1}=\frac{\me e^{rB}\1_{\{A<1\}}}{1-\me
e^{rB}\1_{\{A=1\}}}<\infty
\end{eqnarray*}
and $\mmp\{\widehat{A}_1=1\}=0$ we conclude that $\me
e^{rX}<\infty$ by the previous part of the proof.

\noindent {\sc Proof of \eqref{expx} $\Rightarrow$
\eqref{boundA}}. Assuming that the support of the distribution of
$X$ is unbounded from the right we intend to prove that
$\mmp\{A>1\}>0$ entails $\me e^{rX}=\infty$ for any $r>0$, thereby
providing a contradiction.

In view of $\mmp\{A>1\}>0$ there exist positive constants
$\delta$, $c$ and $\gamma\in (0,1)$ such that
\begin{equation}\label{eq:d1}
\mmp\{A>1+\delta,~ B > -c \} = \gamma.
\end{equation}
Let $(a_i)_{i\in\mn}$ be any sequence satisfying $a_i> 1+\delta$
for all $i\in\mn$. Pick now large enough $m$ such that $m/(m-1)\le
1+\delta$. For the subsequent proof we need the following
inequality
\begin{equation}\label{ineq}
1 + a_1 + a_1a_2 +\ldots + a_1\ldots a_n \le m a_1\ldots a_n
\end{equation}
which will be proved by the mathematical induction. For $n=1$
\eqref{ineq} holds because $m-1\geq 1/(1+\delta)\geq 1/a_1$.
Assuming that \eqref{ineq} holds true for $n=k$ we have
\begin{multline*}
1 + a_1 + a_1a_2 +\ldots + a_1\cdot\ldots\cdot a_k  + a_1\cdot\ldots\cdot a_ka_{k+1}\\
 \le a_1\cdot\ldots\cdot a_k (m + a_{k+1}) \le ma_1\cdot\ldots\cdot a_k a_{k+1} (1/a_{k+1}+1/m) \leq m a_1\cdot\ldots\cdot a_{k+1},
\end{multline*}
by our choice of $m$. Thus, \eqref{ineq} holds for $n=k+1$.

Using \eqref{decomp} with $\tau=n$ gives $X = X_n + \Pi_n
X^{(n)}$. By assumption, $X$ takes arbitrarily large values with
positive probability which implies that $\mmp\{X^{(n)}>mc+1
\}=\mmp\{ X>mc+1 \} = \varepsilon$ for some $\varepsilon>0$ and
all $n\in\mn$. With this at hand, we have for any $n\in\mn$ and
any $r>0$
\begin{eqnarray*}
\me e^{rX}&=& \me e^{r(X_n+\Pi_n X^{(n)})}\\
 &\ge & \me \Big[ e^{r(X_n+\Pi_n X^{(n)})} \1_{\{ A_i > 1+\delta,\, B_i >-c\ \mbox{ for } i=1,\ldots,n \}}\1_{\{ X^{(n)} > mc+1 \}} \Big] \\
 &\ge & \me \Big[ e^{r\big(-c(1+A_1+\ldots + A_1\cdot\ldots\cdot A_{n-1})+ \Pi_n X^{(n)} \big)}\\&\times& \1_{\{ A_i > 1+\delta,\ B_i >  -c\ \mbox{ for } i=1,\ldots,n \}}\1_{\{ X^{(n)} > mc+1 \}} \Big] \\
 &\overset{\eqref{ineq}}{\ge} & \me \Big[ e^{r\big( -mc \Pi_{n-1}+ \Pi_n X^{(n)} \big)} \1_{\{ A_i > 1+\delta,\, B_i >  -c\ \mbox{ for } i=1,\ldots,n \}}\1_{\{ X^{(n)} > mc+1 \}} \Big] \\
 &\ge & \me \Big[ e^{r\big(  \Pi_{n-1}( A_n X^{(n)} - mc)\big)} \1_{\{ A_i > 1+\delta,\ B_i >  -c\ \mbox{ for } i=1,\ldots,n \}}\1_{\{ X^{(n)} > mc+1 \}} \Big] \\
&\ge& e^{r(1+\delta)^{n-1}} \gamma^n \varepsilon.
\end{eqnarray*}
Letting $n$ tend to $\infty$ we obtain $\me e^{rX}=\infty$.

\noindent {\sc Proof of $\eqref{expx}\Rightarrow \eqref{expB}$}.
Assume that $\psi(r)=\me e^{rX}<\infty$ for some $r>0$ and that
the support of the distribution of $X$ is unbounded from the
right. Then $\mmp\{A\leq 1\}=1$ by the previous part of the proof.
Put $c:=\min_{0\leq t\leq r}\,\psi(t)$ and note that $c>0$. Since
$\me e^{rB}\psi(rA)\geq c\me e^{rB}$, the proof is complete in the
case $\mmp\{A=1\}=0$ in view of \eqref{equ}. Suppose now that
$\mmp\{A=1\}\in (0,1)$. In order to check the second inequality in
\eqref{expB} we use once again \eqref{equ} to infer
$$\psi(r)=\me e^{rB}\psi(rA)\1_{\{A<1\}}+\psi(r)\me e^{rB}\1_{\{A=1\}}>\psi(r)\me e^{rB}\1_{\{A=1\}},$$ where the strict inequality
follows from $\mmp\{A<1\}>0$. Now $\me e^{rB}\1_{\{A=1\}}<1$ is a
consequence of the last displayed formula.

It remains to show that $\me e^{sB}<\infty$ for all $s>0$ provided
that the support of the distribution of $X$ is bounded from the
right. If $X\leq 0$ a.s., then $B\leq 0$ a.s.\ whence $\me
e^{sB}\leq 1$ for all $s>0$. Assume now that $\mmp\{X>0\}>0$. This
implies that $\lim_{s\to\infty}\,\psi(s)=\infty$. The latter
together with log-convexity of $\psi$ and its finiteness for all
positive arguments ensures the existence of $s_0\geq 0$ such that
$\psi(s_0)=1$ and $\psi(t)>1$ for any $t>s_0$ (note that $s_0=0$
if $\mmp\{X>0\}=1$, and $s_0>0$ if $\mmp\{X>0\}\in (0,1)$). Using
\eqref{equ} we obtain for $t>s_0$
\begin{eqnarray*}
\psi(t)&=&\me e^{tB}\psi(tA)\1_{\{A\leq 1\}}+\me
e^{tB}\psi(tA)\1_{\{A> 1\}}\geq c_1\me e^{tB}\1_{\{A\leq 1\}}+\me
e^{tB}\1_{\{A> 1\}}\\&\geq& c_1\me e^{tB},
\end{eqnarray*}
where $c_1:=\min_{0\leq s\leq t}\,\psi(s)\in (0,1)$. The proof of
Theorem \ref{main_exp} is complete.
\end{proof}

\begin{proof}[Proof of Theorem \ref{main_exp120}]
We start by showing that \eqref{expx} in combination with
$\mmp\{A<0\}>0$ entails \eqref{boundAA}. Indeed, as a consequence
of \eqref{equ} we infer $$\psi(r)\geq \me
e^{rB}\psi(rA)\1_{\{A<0\}},$$ whence $\psi(-s)<\infty$ for some
$s\in (0,r]$ and thereupon $\me e^{s|X|}\leq
\psi(s)+\psi(-s)<\infty$. Hence, \eqref{boundAA} holds true by
Proposition \ref{AIR}.

Assume now that $\mmp\{A\in (-1, 0)\}=1$. Then $\mmp\{A_1A_2\in
(0,1)\}=1$. Using now decomposition \eqref{decomp} with $\tau=2$
we conclude that $\me e^{rX_2}=\me e^{r(B_1+A_1B_2)}<\infty$
ensures \eqref{expx} by Theorem \ref{main_exp}. In the converse
direction, assuming merely that $A$ is a.s.\ negative, so that
$A_1A_2$ is a.s.\ positive we use again \eqref{decomp} with
$\tau=2$ to obtain that \eqref{expx} entails \eqref{bound_compl}.

Throughout the rest of the proof we assume that $A$ takes values
of both signs with positive probability and that $B$ is a.s.\
nonnegative.

\noindent {\sc Proof of \eqref{expB}  and \eqref{boundAA}
$\Rightarrow$ \eqref{expx}}. We shall use representation
\eqref{decomp} with $$\tau:=\inf\{k\in\mn: \Pi_k>0\}.$$ Observe
that $\mmp\{\tau=1\}=\mmp\{A>0\}=:p$ and
$\mmp\{\tau=k\}=p^{k-2}(1-p)^2$ for $k\geq 2$, whence
$\tau<\infty$ a.s. In view of the first condition in \eqref{expB}
\begin{eqnarray*}
\me e^{rX_\tau}&=&\me
e^{r(B_1+\Pi_1B_2+\ldots+\Pi_{\tau-1}B_\tau)}\\&=&\me
e^{rB_1}\1_{\{A_1>0\}}+\sum_{n\geq 2}\me
e^{r(B_1+\ldots+\Pi_{n-1}B_n)}\1_{\{A_1<0,A_2>0,\ldots,
A_{n-1}>0,A_n<0\}}\\&\leq& \me e^{rB}+\me e^{rB}\sum_{n\geq 2}
\mmp{\{A_2>0,\ldots, A_{n-1}>0,A_n<0\}}\leq 2\me e^{rB}<\infty.
\end{eqnarray*}
Further, $\me e^{rX_\tau}\1_{\{\Pi_\tau=1\}}=\me
e^{rB_1}\1_{\{A_1=1\}}<1$ according to the second condition in
\eqref{expB}. Since $\mmp\{\Pi_\tau\in (0,1]\}=1$ we conclude that
\eqref{expx} holds true by Theorem \ref{main_exp} which applies
because $X$ is also the perpetuity generated by $(\Pi_\tau,
X_\tau)$.

\noindent {\sc Proof of \eqref{expx} $\Rightarrow$ \eqref{expB}
and \eqref{boundAA}}. We shall use $\tau$ as above. Recall that we
have already proved that \eqref{expx} ensures \eqref{boundAA} and
thereupon $\mmp\{\Pi_\tau\in (0,1]\}=1$. Hence, $\me
e^{rX}<\infty$ entails $\me e^{rB_1}\1_{\{A_1=1\}}=\me
e^{rX_\tau}\1_{\{\Pi_\tau=1\}}<1$ by Remark \ref{impo} and $\me
e^{rX_\tau}<\infty$ by Theorem \ref{main_exp}. In particular,
$$\infty>\me e^{rX_\tau}\1_{\{\tau=1\}}=\me e^{rB}\1_{\{A>0\}}$$
and
$$\infty>\me e^{rX_\tau}\1_{\{\tau=2\}}=\me
e^{r(B_1+A_1B_2)}\1_{\{A_1<0, A_2<0\}}\geq \me
e^{rB}\1_{\{A<0\}}\me e^{-rB}\1_{\{A<0\}}$$ whence $\me
e^{rB}<\infty$. The proof of Theorem \ref{main_exp120} is
complete.
\end{proof}

\begin{proof}[Proof of Proposition \ref{main}]
In view of our remark in the introduction we only prove part (I).

For $k\in\mn$, put $(A_k^\ast, B_k^\ast):=(A_k, A_kB_{k+1})$. The
vectors $(A_1^\ast, B_1^\ast), (A_2^\ast, B_2^\ast),\ldots$ are
independent and identically distributed, and
$$A_1B_2+A_1A_2B_3+\ldots=B_1^\ast+A_1^\ast B_2^\ast+A_1^\ast
A_2^\ast B_3^\ast+\ldots=:X^\ast$$ which shows that the left-hand
side is a perpetuity generated by $(A_k^\ast,
B_k^\ast)_{k\in\mn}$. This implies $\me \psi(rA)=\me
e^{r(A_1B_2+A_1A_2B_3+\ldots)}=\me e^{rX^\ast}$.

\noindent {\sc Case} (a). By Theorem \ref{main_exp} $\me
e^{rX^\ast}<\infty$ if, and only if, $\infty>\me e^{rB_1^\ast}=\me
\varphi(rA)$ and $1>\me e^{rB_1^\ast}\1_{\{A_1^\ast=1\}}=\me
e^{rB}\mmp\{A=1\}$. If $\mmp\{A=1\}=0$, the last inequality holds
automatically, whereas if $\mmp\{A=1\}\in (0,1)$ it entails
$\varphi(r)<\infty$ and thereupon $\me \varphi(rA)<\infty$ because
$A\in (0,1]$ a.s.

\noindent {\sc Case} (b). By Theorem \ref{main_exp120}, $\me
e^{rX^\ast}<\infty$ if, and only if, $\infty>\me
e^{r(B_1^\ast+A_1^\ast B_2^\ast)}=\me e^{rA_1(B_2+A_2B_3)}$.

\noindent {\sc Case} (c). According to Remark \ref{neg} and
Proposition \ref{AIR}, $\me e^{rX^\ast}<\infty$ if, and only if,
$\infty>\me e^{r|B_1^\ast|}=\me e^{r|A_1B_2|}$ and
$$\me e^{-rB_1^\ast}\1_{\{A_1^\ast=-1\}}\me e^{rB_1^\ast}\1_{\{A_1^\ast=-1\}}<(1-\me e^{-rB_1^\ast}
\1_{\{A_1^\ast=1\}})(1-\me e^{rB_1^\ast}\1_{\{A_1^\ast=1\}}).$$
The latter is equivalent to
\begin{equation}\label{aux1}
\me e^{-rB}\me e^{rB}[\mmp\{A=-1\}]^2<(1-\me
e^{-rB}\mmp\{A=1\})(1-\me e^{rB}\mmp\{A=1\})
\end{equation}
which entails $$\me e^{r|A_1B_2|}\leq \me e^{r|B|}\leq \me
e^{-rB}+\me e^{rB}<\infty.$$ Thus, $\me e^{rX^\ast}<\infty$ if,
and only if, \eqref{aux1} holds.
\end{proof}

\section{Proof of Theorem \ref{thm:p<1}}\label{12}

Our proof of Theorem \ref{thm:p<1} is based on two auxiliary
results.

\begin{lem}\label{lem:1iter}
Suppose \eqref{cond4} with $c<-1$, \eqref{eq:momB1},
\eqref{eq:tailB} and $\mmp\{A\in (0,1]\}=1$. Let $Y$ be a random
variable independent of $(A,B)$ which satisfies
\begin{equation}\label{x5}
\mmp\{Y>x\}~\sim~ c_Y \mmp\{B>x\},\quad x\to\infty
\end{equation}
for some constant $c_Y>0$. Then $\me f(Y)<\infty$ and
\begin{equation*}
\mmp\{AY+B>x\}~ \sim~ \big(c_Y \me e^{bB}\1_{\{A=1\}} + \me
f(Y)\big)\mmp\{B>x\},\quad x\to\infty.
    \end{equation*}
\end{lem}
\begin{proof}
Fix $\delta \in (0,1)$. In view of
\begin{equation*}
\mmp\{B >(1-\delta)x, Y > \delta x\}~\sim~ a^2c_Y e^{-bx}x^{2c}
(\delta(1-\delta))^c = o(e^{-bx}x^c),\quad x\to\infty
\end{equation*}
and
\begin{equation*}
\{AY+B>x, B\leq (1-\delta) x \} \subseteq \{AY>\delta x\}
\subseteq \{Y>\delta x \} \quad {\rm a.s.}
\end{equation*}
we have
\begin{eqnarray*}
\mmp\{AY+B>x\}&=& \mmp\{AY+B>x, Y \leq \delta x\} + \mmp\{AY+B>x,
B\leq (1-\delta) x\}\\&+&   o(e^{-bx} x^c)=: I_1(x) + I_2(x) +
o(e^{-bx}x^c),\quad x\to\infty.
\end{eqnarray*}

We claim that
\begin{equation*}
\frac{I_1(x)}{\mmp \{B>x\}} = \int_{( -\infty,\,\delta x]}
\frac{\mmp\{B>x-Ay\}}{\mmp\{B>x\}}\mmp\{Y \in {\rm d}y\}~\to~ \me
f(Y)<\infty,\quad x\to\infty.
\end{equation*}
Indeed, this is a consequence of \eqref{eq:tailB} and the Lebesgue
dominated convergence theorem in combination with the following
two facts: (i)
\begin{equation*}
\frac{\mmp \{B>x-Ay\}}{\mmp\{B>x\}} \leq \frac{\mmp
\{B>x-y\}}{\mmp \{B>x\}} \leq M e^{by}\bigg(\frac{x-y}{x}\bigg)^c
\leq M e^{by} (1-\delta )^c
\end{equation*}
for large enough $x$, $y\in [0,\delta x]$ and an appropriate
$M>0$;
 \begin{equation*}
\frac{\mmp \{B>x-Ay\}}{\mmp\{B>x\}} \leq 1
\end{equation*}
for all $x>0$ and all $y<0$, and (ii) $\me e^{bY}<\infty$ which is
an easy consequence of \eqref{cond4} and \eqref{x5}.

Passing to the analysis of $I_2(x)$ we observe that
\begin{equation*}
\lix \frac{\mmp\{uY > x-v\}}{\mmp\{B>x\}}=c_Y e^{bv} \1_{ \{ u=1
\}}
\end{equation*}
for $u\in (0,1]$ and $v\in\mr$. Furthermore,
\begin{equation*}
\frac{\mmp\{uY > x-v\}}{\mmp\{B>x\}} \leq \frac{\mmp\{Y >
x-v\}}{\mmp\{B>x\}} \leq M e^{bv} \bigg(\frac{x-v}{x}\bigg)^c \leq
M e^{bv} \delta^c
\end{equation*}
for large enough $x$, all $u\in [0,1]$, $v\in [0, (1-\delta)x]$
and some appropriate $M>0$, and
\begin{equation*}
\frac{\mmp \{uY>x-v\}}{\mmp\{B>x\}} \leq M_1
\end{equation*}
for large enough $x$, all $u\in [0,1]$, $v<0$ and appropriate
$M_1>0$.

Recalling that $\me e^{bB}<\infty$ we infer
\begin{equation*}
\frac{ I_2(x)}{\mmp\{B>x\}} = \int_{[0,1]\times (-\infty,
(1-\delta)x]}  \frac{\mmp\{uY > x-v\}}{\mmp\{B>x\}} \mmp\{A\in
{\rm d}u,B\in{\rm d}v\}~\to~ c_Y\me e^{bB}\1_{\{A=1\}}
\end{equation*}
as $x\to\infty$ by the dominated convergence theorem.

Combining pieces together finishes the proof of the lemma.
\end{proof}

Apart from Lemma \ref{lem:1iter} we shall use a technique of
stochastic bounds which is a quite commonly used method nowadays.
In the area of perpetuities this approach, as far as we know,
originates from \cite{Grey:1994}. For random variables $U$ and $V$
we shall write $U\leq_{{\rm st}} V$ to indicate that $V$
stochastically dominates $U$, that is, $\mmp\{U>x\}\leq
\mmp\{V>x\}$ for all $x\in\mr$.
\begin{lem}\label{lem:stochbound} Suppose \eqref{cond4} with $c< -1$,
\eqref{eq:momB1}, \eqref{eq:tailB} and $\mmp\{A\in (0,1]\}=1$. On
a possibly enlarged probability space there exists a  nonnegative
random variable $Z$ independent of $(A,B)$ such that
\begin{equation}\label{rel}
\mmp\{Z>x\}~ \sim~ c_Z \mmp\{B>x\},\quad x\to\infty
\end{equation}
for a positive constant $c_Z$ and $AZ+B \leq_{{\rm st}} Z$.
\end{lem}
\begin{proof}
Pick large enough $q>0$ satisfying
\begin{equation*}\label{x2}
\me e^{bB}\1_{\{A=1\}}+\me e^{bB}\1_{\{ B>q \}}<1
\end{equation*}
and then large enough $d>0$ satisfying
\begin{equation*}\label{x1}
e^{bd}>\frac{\mmp\{B\leq q\}}{1-\me e^{bB}\1_{\{A=1\}}-\me
e^{bB}\1_{\{B>q\}}}.
\end{equation*}
Let $B^\prime$ be a copy of $B$ independent of $(A,B)$. Setting
$Y:= (B^\prime + d)\1_{\{B^\prime> q \}}$ we infer
\begin{equation*}
\mmp\{Y> x\}~\sim~ e^{bd} \mmp\{B>x\},\quad x\to\infty.
\end{equation*}

Using Lemma~\ref{lem:1iter} with $c_Y=e^{bd}$ yields
\begin{equation}\label{x3}
\mmp  \{AY+B>x\}~\sim~ \big( e^{bd} \me e^{bB}\1_{ \{A=1\}}+\me
f(Y)\big)\mmp\{B>x\},\quad x\to\infty.
\end{equation}
Since for each $y\geq 0$
$$\frac{\mmp\{B>x-Ay\}}{\mmp\{B>x\}}\leq\frac{\mmp\{B>x-y\}}{\mmp\{B>x\}}~\to~
e^{by},\quad x\to\infty$$ in view of \eqref{cond4} we conclude
that
\begin{equation}\label{bound}
f(y)\leq e^{by},\quad y\geq 0.
\end{equation}
This implies that $$\me f(Y) \leq \me e^{bY} = e^{bd}\me e^{bB}
\1_{\{B>q\}}+\mmp\{B\leq q\},$$ whence
\begin{equation}\label{x4}
e^{bd} \me e^{bB}\1_{\{A=1 \}}+\me f(Y) \leq e^{bd}\me
e^{bB}\1_{\{A=1\}} +e^{bd}\me e^{bB} \1_{\{B>q\}}+\mmp\{B\leq q\}
<e^{bd}
\end{equation}
by the choice of $d$ and $q$. Now \eqref{x3} and \eqref{x4}
together imply that there exists $x_0>0$ such that $\mmp
\{AY+B>x\} \leq \mmp\{Y>x\}$ whenever $x\geq x_0$.

Let $Z$ be a random variable independent of $(A,B)$ with the
distribution $$\mmp\{Z>x\} = \mmp \{Y>x \: | \: Y\geq x_0\}.$$ For
$x\geq x_0$ we have
\begin{equation*}
\mmp\{AZ+B > x\}= \mmp\{AY+B>x \: | \: Y\geq x_0\}\leq \frac{\mmp
\{AY+B>x\}}{\mmp\{Y\geq x_0\}}
        \leq \frac{\mmp \{Y>x\}}{\mmp\{Y\geq x_0\}} = \mmp\{Z>x\}.
    \end{equation*}
For $x <x_0$ $\mmp\{Z>x\}=1$, so that $\mmp\{AZ+B >x\}\leq \mmp
\{Z>x\}$ holds for all $x\in\mr$. The proof of Lemma
\ref{lem:stochbound} is complete.
\end{proof}

\begin{proof}[Proof of Theorem \ref{thm:p<1}]
Let $X_0$ be a nonnegative random variable independent of $(A_n,
B_n)_{n\in\mn}$. The sequence $(X_n)_{n\in\mn_0}$, recursively
defined by the random difference equation
\begin{equation}\label{chain}
X_n= A_n X_{n-1}+B_n,\quad n\in\mn,
\end{equation}
forms a Markov chain. Occasionally, we write $X_n(X_0)$ for $X_n$
to bring out the dependence on $X_0$.

While condition \eqref{cond4} entails $\me \log(1+B^+)<\infty$
which in combination with \eqref{cond2} ensures that $\me
\log(1+|B|)<\infty$ (see the paragraph following formula
\eqref{chf2}), condition $\mmp\{A\in (0,1]\}=1$ together with
\eqref{eq:momB1} guarantees that $\me \log A\in [-\infty, 0)$.
Further, condition \eqref{degen2} obviously holds. Invoking now
Theorem 3.1 (c) in \cite{Goldie+Maller:2000} we conclude that
$X_n$ converges in distribution to the a.s.\ finite $X=\sum_{k\geq
1}\Pi_{k-1}B_k$ as $n\to\infty$ whatever the distribution of
$X_0$. Our plan is to approach the distribution of $X$ from above
and from below by the distributions of $X_n(X_0^{(i)})$,
$n\in\mn_0$, $i=1,2$. By picking appropriate distributions of
$X_0^{(i)}$ we shall be able to provide tight bounds on the
distribution tail of $X$.

\noindent {\sc Upper bound}. Put $X_0=Z$ for a random variable $Z$
as defined in Lemma~\ref{lem:stochbound} which is also independent
of $(A_n, B_n)_{n\in\mn}$. Then
\begin{equation*}
X_1=A_1X_0 +B_1\leq_{{\rm st}}X_0
\end{equation*}
and thereupon
\begin{equation*}
X_{n+1}=A_{n+1}X_n+B_{n+1}\leq_{{\rm st}} A_nX_{n-1}+B_n=X_n,\quad
n\in\mn
\end{equation*}
because $A_k>0$ a.s.\ for $k\in\mn$.

Define a sequence $(c_{X_n})_{n\in\mn_0}$ recursively by
$$c_{X_0}=c_Z,\quad c_{X_{n+1}}=c_{X_n}\me e^{bB}\1_{\{A=1\}}+\me f(X_n),\quad
n\in\mn_0.$$ Note that $$\me f(X)\leq \me f(X_n)\leq \me f(Z)\leq
\me e^{bZ}<\infty,$$ where the first two inequalities hold true
because $f$ is nondecreasing and $(X_n)_{n\in\mn_0}$ is a
stochastically nonincreasing sequence, the third inequality is a
consequence of \eqref{bound}, and the fourth inequality follows
from \eqref{rel} and \eqref{cond4}. Starting with
$$\mmp\{X_0>x\}~\sim~ c_{X_0}\mmp\{B>x\},\quad x\to\infty$$ we use the mathematical induction to
obtain
$$\mmp\{X_n>x\}=\mmp\{A_nX_{n-1}+B_n>x\}~\sim~c_{X_n}\mmp\{B>x\},\quad x\to\infty$$ with
the help of Lemma \ref{lem:1iter}. The latter limit relation
together with the stochastic monotonicity implies that
$(c_{X_n})_{n\in\mn_0}$ is a nonincreasing sequence of positive
numbers which must have a limit $c_X$, say, given by
\begin{equation*}
c_X = \frac{\me f(X)}{1-\me e^{bB}\1_{\{A=1 \}}}.
\end{equation*}
The form of the limit is justified by the fact that the
distributional convergence of $X_n$ to $X$ together with
continuity of the distribution of $X$ (see Theorem 2.1.2 in
\cite{Iksanov:2017} or Theorem 1.3 in
\cite{Alsmeyer+Iksanov+Roesler:2009}) ensures that $f(X_n)$
converges in distribution to $f(X)$ as $n\to\infty$ whence $\lin
\me f(X_n)=\me f(X)$ by the L\'{e}vy monotone convergence theorem.

Since $X\leq_{{\rm st}} X_n$ for each $n\in\mn_0$, we infer
\begin{equation*}
\limsup_{x\to\infty}\frac{\mmp \{X>x\}}{\mmp\{B>x\}}\leq
\limsup_{x\to\infty}\frac{\mmp \{X_n>x\}}{\mmp\{B>x\}}=c_{X_n}
\end{equation*}
for each $n\in\mn_0$ and thereupon
\begin{equation}\label{upper}
\limsup_{x\to\infty}\frac{\mmp \{X>x\}}{\mmp\{B>x\}}\leq c_X.
\end{equation}

\noindent {\sc Lower bound}. We start by noting that
\begin{equation*}
\mmp\{X>x\}=\mmp\{AX+B>x\}\geq\mmp\{AX+B>x, X>0\} \geq
\mmp\{X>0\}\mmp\{B>x\},\quad x\in\mr.
\end{equation*}
Therefore, denoting by $X_0$ a random variable which is
independent of $(A_n,B_n)_{n\in\mn_0}$ and has distribution
$\mmp\{X_0>x\}=\mmp\{X>0\}\mmp\{B>x\}$ for $x\geq 0$ and
$\mmp\{X_0>x\}=\mmp\{X>x\}$ for $x<0$, and arguing in the same way
as in the previous part of the proof we obtain a sequence
$(X_n)_{n\in\mn_0}$ approaching $X$ in distribution such that
$X_n\leq_{{\rm st}}X$ for $n\in\mn_0$. It is worth stating
explicitly that $(X_n)_{n\in\mn_0}$ is not necessarily
stochastically monotone.

Define a sequence $(c^\prime_{X_n})_{n\in\mn_0}$ recursively by
$$c^\prime_{X_0}=\mmp\{X>0\},\quad c^\prime_{X_{n+1}}=c^\prime_{X_n}\me e^{bB}\1_{\{A=1\}}+\me f(X_n),\quad
n\in\mn_0.$$ We claim that
\begin{equation}\label{inter1}
\lin \me f(X_n)=\me f(X)<\infty,
\end{equation}
where the finiteness follows from the previous part of the proof.
Mimicking the argument given in the treatment of the upper bound
we conclude that $f(X_n)$ converges in distribution to $f(X)$ as
$n\to\infty$. Therefore, $\liminf_{n\to\infty}\me f(X_n)\geq\me
f(X)$ by Fatou's lemma. On the other hand, we have $\me f(X_n)\leq
\me f(X)$ for $n\in\mn_0$, and \eqref{inter1} follows.

Now \eqref{inter1} together with $$c^\prime_{X_n}=(\mmp\{X>0\}\me
e^{bB}\1_{\{A=1\}})^n+\sum_{k=0}^{n-1} (\me
e^{bB}\1_{\{A=1\}})^{n-k-1} \me f(X_k)$$ for $n\in\mn$ ensures
that $c^\prime_X:=\lin c^\prime_{X_n}$ exists and
$$c^\prime_X=\frac{\me f(X)}{1-\me e^{bB}\1_{\{A=1\}}}=c_X.$$ The
same argument as in the previous part of the proof enables us to
conclude that
\begin{equation*}
\liminf_{x\to\infty}\frac{\mmp \{X>x\}}{\mmp \{B>x\}}\geq
c^\prime_{X_n}
\end{equation*}
for each $n\in\mn_0$, whence
\begin{equation}\label{lower}
\liminf_{x\to\infty}\frac{\mmp \{X>x\}}{\mmp \{B>x\}}\geq c_X.
\end{equation}
A combination of \eqref{upper} and \eqref{lower} yields
\eqref{res}. The proof of Theorem \ref{thm:p<1} is complete.
\end{proof}

As was announced in Remark \ref{refe} we are now discussing
similarities between the preceding proof and the proof of Theorem
3 in \cite{Palmowski+Zwart:2007}. First, our Lemma \ref{lem:1iter}
resembles Lemma 2 in \cite{Palmowski+Zwart:2007}. Secondly, the
random variables $Z$ and $Y_1^{\downarrow}$ appearing in our Lemma
\ref{lem:stochbound} and the proof of Theorem 3 in
\cite{Palmowski+Zwart:2007}, respectively, serve analogous
purposes.

\section{Proof of Theorem \ref{main2}}\label{13}

Recall that $\Psi(t)=\me e^{\iii t X}$, $t\in\mr$ satisfies
\eqref{chf}. Using
\begin{align*}
\int_0^\infty e^{\iii t y}\mmp\{B>y\}{\rm d}y& -\int_{-\infty}^{0}e^{\iii t y}\mmp\{B\leq y\}{\rm d}y\\
&=\me \left(\left[\int_0^B e^{\iii t y}{\rm d}y\right]\1_{\{B>0\}}\right)-\me\left( \left[\int_B^0 e^{\iii t y}{\rm d}y\right]\1_{\{B\leq 0\}}\right)\\
&=\me \left(\frac{e^{\iii t B}-1}{\iii t}\1_{\{B>0\}}\right)-\me
\left(\frac{1-e^{\iii t B}}{\iii t}\1_{\{B\leq
0\}}\right)=\frac{\Phi(t)-1}{\iii t},
\end{align*}
we obtain an equivalent form of \eqref{chf}
\begin{align*}
\Psi(t)&=\Phi(t)\exp\left(\iii \lambda\int_0^t \left[\int_0^{\infty}e^{\iii u y}\mmp\{B>y\}{\rm d}y -\int_{-\infty}^{0}e^{\iii u y}\mmp\{B\leq y\}{\rm d}y\right]{\rm d}u\right)\\
&=\Phi(t)\exp\left(\lambda \left[\int_0^{\infty}\frac{e^{\iii
ty}-1}{y}\mmp\{B>y\}{\rm d}y-\int_{-\infty}^0\frac{e^{\iii t
y}-1}{y}\mmp\{B\leq y\}{\rm d}y\right]\right)
\end{align*}
for $t\in\mr$. In view of \eqref{dec1} this can be further
represented as
\begin{eqnarray}
&&\Psi(t)\exp\left(\lambda \int_0^\infty\frac{e^{\iii ty}-1}{y}r^-(y){\rm d}y\right)\label{eq1}\\
&=&\Phi(t)\left(\frac{b}{b-\iii
t}\right)^{C\lambda}\exp\left(\lambda \int_0^\infty \frac{e^{\iii
ty}-1}{y}r^{+}(y){\rm d}y\right)\notag\\&\times&\exp\left(\lambda
\int_0^\infty\frac{e^{-\iii ty}-1}{y}\mmp\{-B\geq y\}{\rm
d}y\right)\notag
\end{eqnarray}
for $t\in\mr$. Let $Z_1$, $Z_2$ and $Z_3$ be infinitely divisible
nonnegative random variables with zero drifts and the L\'{e}vy
measures $\nu_1({\rm d}y)=y^{-1}r^-(y)\1_{(0,\infty)}(y){\rm d}y$,
$\nu_2({\rm d}y)=y^{-1}r^+(y)\1_{(0,\infty)}(y){\rm d}y$ and
$\nu_3({\rm d}y)=y^{-1}\mmp\{-B\geq y\}\1_{(0,\infty)}(y){\rm
d}y$, respectively. Let $V$ be a random variable with a
$\gamma(C\lambda, b)$ distribution. Assume that $Z_1$ is
independent of $X$ and that $B$, $V$, $Z_2$ and $Z_3$ are mutually
independent. Equality \eqref{eq1} tells us that $X+Z_1$ has the
same distribution as $B+V+Z_2-Z_3$. We claim that
\begin{equation}\label{asym}
\mmp\{X+Z_1>x\}=\mmp\{B+V+Z_2-Z_3>x\}~\sim~ C\me
e^{b(Z_2-Z_3)}\frac{(bx)^{C\lambda}}{\Gamma(C\lambda+1)}e^{-bx}
\end{equation}
as $x\to\infty$, where $\me e^{b(Z_1-Z_2)}\leq \me
e^{bZ_1}<\infty$ by virtue of the first condition in
\eqref{r_1_integreable}.

\noindent {\sc Proof of \eqref{asym}}. By \eqref{dec1} and
\eqref{cond1}, $\mmp\{B>x\}\sim Ce^{-bx}$ as $x\to\infty$. Hence,
$$\mmp\{B+V>x\}\sim
\frac{C(bx)^{C\lambda}}{\Gamma(C\lambda+1)}e^{-bx},\quad
x\to\infty$$ by Lemma 7.1 (iii) in \cite{Rootzen:1986} (in the
notation of \cite{Rootzen:1986} we set $Y_1:=bB$ and $Y_2:=bZ$).
According to an extension of Breiman's lemma (Proposition 2.1 in
\cite{Denisov+Zwart:2007}) relation \eqref{asym} follows provided
that the following conditions hold: (a) $\me
e^{b(Z_1-Z_2)}<\infty$; (b) $x^b\mmp\{e^{Z_1-Z_2}>x\}=o(h(x))$ as
$x\to\infty$, where $h(x):=x^b\mmp\{e^{B+V}>x\}$ for $x\geq 0$;
(c) ${\lim\sup}_{x\to\infty}\frac{\sup_{1\leq y\leq
x}\,h(y)}{h(x)}<\infty$. We already know that (a) holds which
particularly implies that
$\lim_{x\to\infty}\,x^b\mmp\{e^{Z_1-Z_2}>x\}=0$. While this in
combination with $h(x)\sim (Cb^{C\lambda}/\Gamma(C\lambda+1))(\log
x)^{C\lambda}$ proves (b), the last asymptotic relation alone
secures (c). The proof of \eqref{asym} is complete.

With \eqref{asym} at hand we infer
$$
\mmp\{X>x\} \sim \frac{\me e^{b(Y_2-Y_3)}}{\me
e^{bY_1}}\frac{C(bx)^{\lambda C}}{\Gamma(\lambda
C+1)}e^{-bx}=Kx^{\lambda C}e^{-bx},\quad x\to\infty
$$
by Corollary 4.3 (ii) in \cite{Jacobsen+etal:2009} which is
applicable because the distribution of $Y_1$ is infinitely
divisible and $\me e^{(b+\varepsilon)Y_1}<\infty$ by the second
part of \eqref{r_1_integreable}. The proof of Theorem \ref{main2}
is complete.

\acks D. Buraczewski and P. Dyszewski were partially supported by
the National Science Center, Poland (Sonata Bis, grant number
DEC-2014/14/E/ST1/00588). The work of A. Marynych was supported by
the Alexander von Humboldt Foundation. The authors thank two
anonymous referees for pointing out several oversights in the
original version and many other useful comments which helped
improving the presentation of this work.


\begin{thebibliography}{99}
\footnotesize

\bibitem{Alsmeyer+Dyszewski:2016+} {\sc Alsmeyer, G. and Dyszewski, P.} (2017). Thin tails of fixed points of
the non\-ho\-mo\-ge\-ne\-ous smoothing transform. {\em Stoch.
Proc. Appl.} {\bf 127,} 3014--3041.

\bibitem{Alsmeyer+Iksanov+Roesler:2009} {\sc Alsmeyer, G., Iksanov, A. and  R\"{o}sler, U.} (2009). On distributional properties of
perpetuities. {\em J. Theoret. Probab.} {\bf 22,} 666--682.

\bibitem{Bondesson:1992} {\sc Bondesson, L.} (1992). {\em Generalized gamma convolutions and related classes of distributions and densities}.
Lecture Notes in Statistics, {\bf 76}. Springer-Verlag, New York.

\bibitem{Breiman:1965} {\sc Breiman, L.} (1965). On some limit theorems similar to the
arc-sin law. {\em Theory Probab. Appl.} {\bf 10,} 323--331.

\bibitem{Buraczewski+Damek+Mikosch:2016} {\sc Buraczewski, D., Damek, E. and Mikosch, T.} (2016). {\em Stochastic models with power-law
tails: the equation $X=AX+B$}. Springer, Cham.

\bibitem{Cline+Samorodnitsky:1994} {\sc Cline, D. and Samorodnitsky, G.} (1994). Subexponentiality of the
product of independent random variables. {\em Stoch. Proc. Appl.}
{\bf 49,} 75--98.

\bibitem{Damek+Dyszewski:2017} {\sc Damek, E. and Dyszewski, P.} (2017). Iterated random functions and regularly varying tails.
Preprint available at {\tt https://arxiv.org/abs/1706.03876}

\bibitem{Damek+Kolodziejek:2017} {\sc Damek, E. and Ko{\l}odziejek, B.} (2017). Stochastic recursions: between Kesten's and Grey's assumptions.
Preprint available at {\tt https://arxiv.org/abs/1701.02625}

\bibitem{Denisov+Zwart:2007} {\sc Denisov, D. and Zwart, B.} (2007). On a theorem of Breiman and a class of random
difference equations. {\em J. Appl. Probab.} {\bf 44,} 1031--1046.

\bibitem{Dyszewski:2016} {\sc Dyszewski, P.} (2016). Iterated random functions and slowly varying tails.
{\em Stoch. Proc. Appl.} {\bf 126,} 392--413.

\bibitem{Goldie:1991} {\sc Goldie, C.~M.} (1991). Implicit renewal theory and tails of solutions of random equations.
{\em Ann. Appl. Probab.} {\bf 1,} 126--166.

\bibitem{Goldie+Grubel:1996} {\sc Goldie, C.~M. and Gr\"{u}bel, R.} (1996). Perpetuities with thin tails. {\em Adv.
Appl. Probab.} {\bf 28,} 463--480.

\bibitem{Goldie+Maller:2000} {\sc Goldie, C.~M. and Maller, R.~A.} (2000). Stability of perpetuities. {\em Ann. Probab.} {\bf 28,} 1195--1218.

\bibitem{Gradshteyn+Ryzhik:2000} {\sc Gradshteyn, I.~S. and Ryzhik, I.~M.} (2000). {\em Table of integrals, series, and products}.
Academic Press, San Diego.

\bibitem{Grey:1994} {\sc Grey, D.} (1994). Regular variation in the tail behaviour of solutions of random difference equations. {\em
Ann. Appl. Probab.} {\bf 4,} 169--183.

\bibitem{Grincevicius:1975} {\sc Grincevi\v{c}ius, A.~K.} (1975). One limit distribution for a random walk on the line.
{\em Lithuanian Math. J.} {\bf 15,} 580--589.

\bibitem{Hitczenko:2010} {\sc Hitczenko, P.} (2010). On tails of
perpetuities. {\em J. Appl. Probab.} {\bf 47,} 1191--1194.

\bibitem{Hitczenko+Wesolowski:2009} {\sc Hitczenko, P. and  Weso{\l}owski, J.} (2009). Perpetuities with thin tails revisited.
{\em Ann. Appl. Probab.} {\bf 19,} 2080--2101. Erratum: {\em Ann.
Appl. Probab.} {\bf 20,} (2010), 1177.

\bibitem{Iksanov:2002} {\sc Iksanov, O.~M.} (2002). On positive distributions of the class $L$ of self-decomposable
laws. {\em Theor. Probab. Math. Statist.} {\bf 64,} 51--61.

\bibitem{Iksanov:2017} {\sc Iksanov, A.} (2016). {\em Renewal theory for perturbed random walks and similar
processes}. Birkh\"{a}user, Basel.

\bibitem{Iksanov+Jurek:2003} {\sc Iksanov, A.~M. and Jurek, Z.~J.} (2003).
Shot noise distributions and selfdecomposability. {\em Stoch.
Analysis Appl.} {\bf 21,} 593--609.

\bibitem{Jacobsen+etal:2009} {\sc Jacobsen, M., Mikosch, T., Rosi\'{n}ski, J. and Samorodnitsky, G.}
(2009). Inverse problems for regular variation of linear filters,
a cancellation property for $\sigma$-finite measures and
identification of stable laws. {\em Ann. Appl. Probab.} {\bf 19,}
210--242.

\bibitem{Jurek+Vervaat:1983} {\sc Jurek, Z.~J. and Vervaat, W.} (1983). An integral representation for selfdecomposable Banach space valued random
variables. {\em Z. Wahrscheinlichkeitstheorie Verw. Geb.} {\bf
62,} 247--262.

\bibitem{Kesten:1973} {\sc Kesten, H.} (1973). Random difference equations and renewal theory for products of random
matrices. {\em Acta Math.} {\bf 131,} 207--248.

\bibitem{Kevei:2016} {\sc Kevei, P.} (2016). A note on the Kesten-Grincevi{\v{c}}ius-Goldie theorem.
{\em Electron. Commun. Probab.} {\bf 21,} no. 51, 1--12.

\bibitem{Kevei:2017} {\sc Kevei, P.} (2017). Implicit renewal theory in the arithmetic case. {\em J. Appl. Probab.} {\bf 54,} 732--749.

\bibitem{Kolodziejek:2016+} {\sc Ko{\l}odziejek, B.} (2017). Logarithmic tails of sums of products of
positive random variables bounded by one. {\em Ann. Appl. Probab.}
{\bf 27,} 1171--1189.

\bibitem{Konstantinides+Ng+Tang:2010}  {\sc Konstantinides, D.~G., Ng, K.~W and Tang, Q.} (2010). The probabilities of absolute ruin
in the renewal risk model with constant force of interest. {\em J.
Appl. Probab.} {\bf 47,} 323--334.

\bibitem{Maulik+Zwart:2006} {\sc Maulik, K. and  Zwart, B.} (2006). Tail asymptotics for exponential functionals of L\'{e}vy
processes. {\em Stoch. Proc. Appl.} {\bf 116,} 156--177.

\bibitem{Palmowski+Zwart:2007} {\sc Palmowski, Z. and Zwart, B.} (2007). Tail asymptotics of the supremum of a regenerative process.
{\em J. Appl. Probab.} {\bf 44,} 349--365.

\bibitem{Rootzen:1986} {\sc Rootz\'{e}n, H.} (1986). Extreme value theory for moving average processes. {\em Ann. Probab.} {\bf 14,} 612--652.

\bibitem{Takacz:1955} {\sc Tak\'{a}cs, L.} (1955). Stochastic processes connected with certain physical recording apparatuses.
{\em Acta Math. Acad. Sci. Hungar.} {\bf 6,} 363--380.

\bibitem{Vervaat:1979} {\sc Vervaat, W.} (1979). On a stochastic difference equation and
a representation of nonnegative infinitely divisible random
variables. {\em Adv. Appl. Probab.} {\bf 11,} 750--783.


\end{thebibliography}
\end{document}